\documentclass[11pt]{article}
\usepackage{amsmath,amsthm,verbatim,amssymb,amsfonts,amscd, graphicx,color}
\usepackage{hyperref}
\usepackage{parskip}

\usepackage[utf8]{inputenc}
\usepackage[T1]{fontenc}

\theoremstyle{plain}
\newtheorem{theorem}{Theorem}[section]
\newtheorem{lemma}{Lemma}[section]

\newtheorem{proposition}{Proposition}[section]

\newtheorem*{claim*}{Claim}
\newtheorem*{lemma*}{Lemma}
\newtheorem*{theorem*}{Theorem}

\theoremstyle{definition}
\newtheorem{definition}{Definition}[section]

\theoremstyle{remark}

\DeclareMathOperator{\Hessian}{Hess}

\DeclareMathOperator{\support}{supp}

\DeclareMathOperator{\Dom}{Dom}

\renewcommand{\Im}{\operatorname{Im}}
\renewcommand{\Re}{\operatorname{Re}}
\renewcommand{\i}{\operatorname{\sqrt{-1}}}

\allowdisplaybreaks

\usepackage[
backend=bibtex,
natbib=true,
url=false,
]{biblatex}
\IfFileExists{mybibliography.bib}
{\addbibresource{mybibliography.bib}}
{}
\IfFileExists{../bib/mybibliography.bib}
{\addbibresource{../bib/mybibliography.bib}}
{}

\begin{document}
	
	\title{The $\bar{\partial}$-Neumann operator with the Sobolev norm of integer orders}
	
	\author{
		Phillip Harrington\\
		psharrin@uark.edu
		\and
		Bingyuan Liu\\
		bl016@uark.edu
	}
	
	\date{\today}

	\maketitle
	
	\begin{abstract}
Let $\Omega\subset\mathbb{C}^m$  be a bounded pseudoconvex domain with smooth boundary. For each $k\in\mathbb{N}$, we give a sufficient condition to estimate the $\bar\partial$-Neumann operator in the Sobolev space $W^k(\Omega)$. The key feature of our results is a precise formula for $k$ in terms of the geometry of the boundary of $\Omega$.
	\end{abstract}
\section{Introduction}
Let $\Omega\subset \mathbb{C}^m$ be a bounded pseudoconvex domain with smooth boundary. Solving the $\bar{\partial}$-equation on $\Omega$ has a long history. It has many applications in various fields including differential geometry, algebraic geometry and partial differential equations.

The $L^2$ theory for $\bar{\partial}$ originated in work of Andreotti--Vesentini (\cite{AV61} and \cite{AV65}) and H\"{o}rmander (\cite{Ho65}). They found the $\bar{\partial}$-equation is weakly solvable on pseudoconvex domains in $\mathbb{C}^m$. This led to successes in several complex variables, algebraic geometry and partial differential equations. As a byproduct, the $\bar\partial$-Neumann operator, an inverse of the complex Laplacian $\Box$ with a non-coercive boundary condition, is also bounded in $L^2$ for pseudoconvex domains.

In the meantime, Kohn (\cite{Ko63} and \cite{Ko64}) studied the regularity of the $\bar{\partial}$-Neumann operator in Sobolev spaces. He found the $\bar{\partial}$-Neumann operator is also bounded in the Sobolev space $W^k(\Omega)$ for all $k>0$ on bounded strongly pseudoconvex domains with smooth boundaries. Indeed, he found on a strongly pseudoconvex domains, one has the inequality \[\|u\|_{W^{k+1}}\lesssim\|\Box u\|_{W^k}\] for all $k>0$ and $u$ in the domain of $\Box$. Subelliptic estimates of this type with weaker subelliptic gains have been studied on pseudoconvex domains of finite type by Catlin (\cite{Ca87}) and D'Angelo (\cite{DA79} and \cite{DA82}).

However, the extension of the inequality \[\|u\|_{W^k}\lesssim\|\Box u\|_{W^k} \] to infinite type pseudoconvex domains appears later. In \cite{KN65}, Kohn and Nirenberg showed that it suffices for the $\bar\partial$-Neumann operator to be compact.  In the 1990s, Boas and Straube (\cite{BS90}, \cite{BS91}, \cite{BS91b} and \cite{BS93}) found a condition which makes the inequality  \[\|u\|_{W^k}\lesssim\|\Box u\|_{W^k} \] holds for all $k>0$ (see also the work of Chen \cite{Ch91}). Around the same time, Barrett (\cite{Ba92}) discovered the aforementioned inequality cannot hold on the worm domains of Diederich and Forn{\ae}ss, a family of smooth, bounded pseudoconvex domains which do not satisfy the condition of Boas and Straube. The condition of Boas and Straube turns to be useful in many subjects and this condition was recently improved by the first author (\cite{Ha11}) and Straube (\cite{St08}).

In this article, we will give a sufficient condition for \[\|u\|_{W^k}\lesssim\|\Box u\|_{W^k} \] to hold for a fixed positive integer $k\in\mathbb{N}$. The key feature of this paper is a precise formula for $k$ in terms of the geometry of the boundary of $\Omega$. Our main theorems are Theorem \ref{maintheorem} and Theorem \ref{maintheorem2}. Our method is based on the fundamental work of Boas and Straube \cite{BS91}, \cite{BS91b}, \cite{BS93} and Chen \cite{Ch91}. We note that Boas and Straube's method involves first estimating the Bergman projection, since they have already shown that these estimates are equivalent to estimates for the $\bar\partial$-Neumann problem \cite{BS90}.  Our approach follows more closely the method of Chen \cite{Ch91}, which directly estimates the $\bar\partial$-Neumann operator without the intermediate result on the Bergman projection.  The reader is referred to the books of Chen and Shaw \cite{CS01} and Straube \cite{St10}.

\section{Preliminaries}

\subsection{Notation and Main Results}

For $\Omega\subset\mathbb{C}^m$, we denote the set of weakly pseudoconvex points in $\partial\Omega$ by $\Sigma$. Let $\lbrace U_\alpha\rbrace_{\alpha=1}^M$ be open sets of which the union covers $\Sigma$. Let $\lbrace L_j\rbrace_{j=1}^{m-1}$ be $(1, 0)$ tangential vector fields with smooth coefficients in $U_\alpha$ such that $\lbrace L_j|_{\partial\Omega}\rbrace_{j=1}^{m-1}$ are tangential to $\partial\Omega$. Let $L_m=(\sum|\frac{\partial\delta}{\partial z_i}|^2)^{-1}\sum \frac{\partial\delta}{\partial \bar{z}_i}\frac{\partial}{\partial z_i}$ be the $(1,0)$ complex normal vector in $U_\alpha$, where $\delta$ denotes the signed distance function for $\Omega$. We also assume $\lbrace \sqrt{2}L_j\rbrace_{j=1}^{m}$ form an orthonormal basis for $T^{1,0}(U_\alpha)$.

For the following, we will fix $\alpha$ and assume that all forms are supported in $U_\alpha$.  We will show in the proof of Theorem \ref{maintheorem} that a partition of unity can be used to combine these local estimates.

Let $g(\cdot,\cdot)$ denotes the Euclidean metric. For example \[g(f\omega_{\overline{L}_1}\wedge\omega_{\overline{L}_2},h\omega_{\overline{L}_1}\wedge\omega_{\overline{L}_2})=f\bar{h}g(\omega_{\overline{L}_1}\wedge\omega_{\overline{L}_2},\omega_{\overline{L}_1}\wedge\omega_{\overline{L}_2})=\frac{f\bar{h}}{2^2}.\] Angle brackets denote the $L^2$ inner product, i.e., $<u,v>=<u,v>_{L^2}=\int_{\Omega}g(u,v)\,dV$. We use $\|\cdot\|$ to denote the $L^2$ norm. Let $\nabla$ denote the Levi-Civita connection with respect to the Euclidean metric.
Also, $\|u\|^2_k=\int_{\Omega}g(\nabla^k u,\nabla^k u)\,dV$ is slightly different from the norm for the Sobolev space of order $k$. The latter is defined to be \[\|u\|^2_{W^k}=\sum_{j=0}^{k}\int_{\Omega}g(\nabla^ju,\nabla^ju)\,dV.\]

Let $u$ be a $(0, q)$-form for $1\leq q\leq m$. Let ${\sum_I}'$ denote the sum over all $q$-tuples $I$ with increasing indices.  We may express $u$ in local coordinate chart as:
\[u={\sum_{I}}'f_I\omega_{\overline{L}_I}=\frac{1}{2^q}{\sum_{I}}'g(u,\omega_{\overline{L}_I})\omega_{\overline{L}_I}\]
In this notation, we may compute
\[\bar{\partial}u=\frac{1}{2^q}{\sum_{I}}' \sum_{i\notin I} g(\nabla_{\overline{L}_i}u,\omega_{\overline{L}_I})\omega_{\overline{L}_i}\wedge\omega_{\overline{L}_I},\]
and
\[\bar{\partial}^* u=-\frac{1}{2^{q-1}}{\sum_{I}}'\sum_{i\in I} g(\nabla_{L_i}u, \omega_{\overline{L}_I})\overline{L}_i\lrcorner\omega_{\overline{L}_I}.\]

We observe that $u\in\Dom(\bar{\partial}^*)$ if and only if $g(u, \omega_{\overline{L}_J}\wedge\omega_{\overline{L}_m})=0$ on $\partial\Omega$, for all arbitrary $q-1$-tuples $J$ with increasing indices.

Moreover, it is easy to check that

\[\nabla_{L_i}\omega_{\overline{L}_j}=2\sum_{i,j,k}g(\nabla_{L_i}L_j, L_k)\omega_{\overline{L}_k}\qquad\text{and}\qquad\nabla_{\overline{L}_i}\omega_{\overline{L}_j}=2\sum_{i,j,k}g(\nabla_{\overline{L}_i}L_j, L_k)\omega_{\overline{L}_k}.\]

In general, $\nabla_{L_m-\overline{L}_m}u$ is not in $\Dom(\bar{\partial}^*)\cap C^\infty_{(0,1)}(\overline{\Omega})$. For this, we define the following notation $\nabla^c$.

\begin{definition}
\label{defn:special_connection}
	Let $u=\frac{1}{2^q}{\sum_{I}}'g(u,\omega_{\overline{L}_I})\omega_{\overline{L}_I}$. Then \[\nabla^c_{L_m-\overline{L}_m} u:=\frac{1}{2^q}{\sum_{I}}'(L_m-\overline{L_m})g(u,\omega_{\overline{L}_I})\omega_{\overline{L}_I}.\]
\end{definition}
In other words, $\nabla^c$ only acts on the coefficients of $u$ with respect to $\omega_{\overline{L}_I}$.
We also observe that \[\nabla_{L_m-\overline{L}_m}u=\nabla^c_{L-\overline{L}_m}u+\frac{1}{2^q}{\sum_{I}}'g(u,\omega_{\overline{L}_I})\nabla_{L_m-\overline{L}_m}\omega_{\overline{L}_I}.\]

Moreover, $u\in\Dom(\bar{\partial}^*)\cap C^\infty_{(0,q)}(\overline{\Omega})$ implies $\nabla_{L_m-\overline{L}_m}^cu\in\Dom(\bar{\partial}^*)\cap C^\infty_{(0,q)}(\overline{\Omega})$.

It is easy to check that \[(L_m-\overline{L}_m)g(u,v)=g(\nabla^c_{L_m-\overline{L}_m}u, v)+g(u,\nabla^c_{\overline{L}_m-L_m}v)\]

With this notation in place, we are able to state our main result:
\begin{theorem}\label{maintheorem2}
	Let $\Omega$ be a bounded pseudoconvex domain with smooth boundary in $\mathbb{C}^m$ and $1\leq q\leq m$.  Assume there exists a real-valued smooth function $\psi$ defined in a neighborhood $U$ of $\Sigma$ in $\partial\Omega$ so that \[\sup_{\Sigma}|L\psi-2g(\nabla_{L}L_m,L_m)|<\frac{1}{km\sqrt{B{m\choose q}}},\] for all unit $(1,0)$ vector fields $L$ in a neighborhood of $\partial\Omega$ where $k\in\mathbb{N}$ and $B$ is the constant from Theorem \ref{basic}. Then the $\bar{\partial}$-Neumann operator for $(0,q)$-forms is bounded from $W^k$ to $W^k$.
\end{theorem}
We will see in Section \ref{sec:commutators_first_derivatives} that our weight function $\psi$ allows the construction of a $(1,0)$-vector field $e^\psi L_m$ related to those studied by Boas and Straube (see \cite{BS99}, for example), or a real tangential vector field $e^{\psi}(L_m-\overline{L}_m)$ related to those studied by Chen \cite{Ch91}.  In contrast to those papers, note that we do not require any uniform bound on $\psi$; it has already been noted by the first author in \cite{Ha11} that such a restriction may not be necessary.  We accomplish this partly by requiring $\psi$ to be real-valued, in contrast to the condition studied by Boas and Straube; it remains to be seen if our method can be generalized to the case in which the imaginary part of $\psi$ is sufficiently small.

\subsection{Foundational Results}

The following estimate is called the basic estimate in many contexts; see the survey by Boas and Straube \citep{BS99} for details.

\begin{theorem}[Hörmander \cite{Ho65}]\label{basic}
	Let $\Omega$ be a bounded pseudoconvex domain in $\mathbb{C}^m$ with smooth boundary. For any arbitrary $(0,q)$-form  $u\in\Dom(\bar\partial)\cap\Dom(\bar{\partial}^*)$, $1\leq q \leq n$, we have that
	\[\|u\|^2\leq B (\|\bar{\partial}u\|^2 +\|\bar{\partial}^*u\|^2)\] for some $B>0$. Moreover, if $u\in\Dom(\Box)$, we have that \[\|u\|\leq B\|\Box u\|.\]
\end{theorem}
This result is a consequence of the Morrey-Kohn-H\"ormander identity; see Proposition 4.3.1 in \cite{CS01} for details.  H\"ormander uses this estimate to prove closed range and solvability for the $\bar\partial$-operator in the $L^2$ sense.  The optimal constant $B$ represents the reciprocal of the smallest eigenvalue of $\Box$, or equivalently, the largest eigenvalue of the $\bar\partial$-Neumann operator.  H\"ormander's proof in \cite{Ho65} demonstrates that this optimal constant is bounded above by $\frac{e}{q}\sup_{z,z'\in\Omega}|z-z'|^2$ (see Theorem 4.3.4 in \cite{CS01}).

For strongly pseudoconvex domains, it is known that the Sobolev estimate of arbitrary order hold. More generally, we have the following pseudolocal estimate.

\begin{theorem}[See Theorem 3.6 in Straube \citep{St10}]\label{pseudolocal}
	Let $\Omega$ be a bounded strongly pseudoconvex domain in $\mathbb{C}^m$ with smooth boundary. Fix a cutoff function $\phi\in C^\infty_0(\mathbb{C}^m)$. For every $k>0$, the following estimates hold for arbitrary $(0,q)$-forms  $u\in\Dom(\Box)$, $(1\leq q \leq n)$:
	\[\|\phi u\|_{W^{k+1}}\lesssim\|\Box u\|_{W^k}.\]
\end{theorem}
Observe that this follows from (3.62) in \cite{St10} with $\varphi_2\equiv 1$.  To make use of this, we will let $\chi$ be an arbitrary cutoff function supported in a neighborhood of $\Sigma$ and $\chi\equiv 1$ on $\Sigma$. Let $\phi:=1-\chi$.  We have that
\begin{equation}\label{ineq}
	\|u\|_{W^k}\leq\|\chi u\|_{W^k}+\|\phi u\|_{W^k}\leq \|\chi u\|_{W^k}+C\|\Box u\|_{W^k}.
\end{equation}

Consider \[\|\chi u\|_k^2=\int_{\Omega}g(\nabla^k\chi u,\nabla^k\chi u)\,dV\lesssim\|\Box u\|_{W^{k-1}}+\int_{\Omega}|\chi|^2g(\nabla^k u,\nabla^k u)\,dV.\] If we are able to show that $\int_{\Omega}|\chi|^2g(\nabla^k u,\nabla^k u)\,dV\lesssim\|\Box u\|_{W^k}$, we are done.

By Lemma 5.6 in Straube \citep{St10}, we can see that all derivatives are benign except for those in the direction of $L_m-\overline{L}_m$; see Section \ref{benign}. Hence, we just need to show that $\int_{\Omega}|\chi|^2g((\nabla_T^c)^k u,(\nabla_T^c)^k u)\,dV\lesssim\|\Box u\|_{W^k}$, for $T=e^\psi (L_m-\overline{L}_m)$, where $\psi$ is some weight function.

	\section{Benign Derivatives}\label{benign}
	
	\begin{proposition}
\label{prop:benign_derivative_setup}
		Let $\Omega\subset\mathbb{C}^m$ be a bounded domain with smooth boundary.  For every $k\in\mathbb{N}$, there exists a constant $C_k$ such that
		\begin{equation}
		\label{eq:benign_derivative_estimate}
		\|\bar{\partial} u\|^2_{W^k}+\|\bar{\partial}^* u\|^2_{W^k}\leq C_k(\|\Box u\|_{W^k}^2+\|u\|_{W^k}^2)
		\end{equation}
		for all $u\in C^\infty_{(0,q)}(\overline\Omega)\cap\Dom(\Box)$.
	\end{proposition}
	
	\begin{proof}
		We proceed by induction on $k$.  When $k=0$, this follows for $C_0=1$ because
		\[
		\|\bar{\partial} u\|^2+	\|\bar{\partial}^* u\|^2=(\Box u,u).
		\]
		We henceforth assume $k\geq 1$.
		
		Throughout the following $D^\ell_\tau$ will denote a generic differential operator of order $\ell$ that is tangential on $b\Omega$ that acts on coefficients of $u$, with respect to our local boundary coordinates, as in Definition \ref{defn:special_connection}.  It is clear that $D^\ell_\tau$ preserves $\Dom\bar{\partial}^*$.  Let $\zeta=L_m+\overline{L}_m$ denote a vector field that is equal to the outward unit normal on $b\Omega$.  To estimate the Sobolev norm of order $k$, we will need to estimate all operators of the form $D_\tau^{k-j}(\nabla^c_\zeta)^j$, where $1\leq j\leq k$.

		When $j\geq 2$, then we may use, e.g., (3.42) in \cite{St10} to rewrite $(\nabla^c_\zeta)^2$ as a linear combination of $\Box$, an operator of the form $D_\tau^2$, and lower order terms.  We note that $\bar{\partial} u\in\Dom(\Box)$, but $\bar{\partial}^* u$ is not necessarily in the domain of $\Box$, so we will need to interpret $\Box$ formally when we write $\Box\bar{\partial}^* u$.  Hence,
		\begin{multline*}
		\|D_\tau^{k-j}(\nabla^c_\zeta)^j\bar{\partial} u\|^2+\|D_\tau^{k-j}(\nabla^c_\zeta)^j\bar{\partial}^* u\|^2\leq \\ O\left(\|D_\tau^{k-j}(\nabla^c_\zeta)^{j-2}\Box\bar{\partial} u\|^2+\|D_\tau^{k-j}(\nabla^c_\zeta)^{j-2}\Box\bar{\partial}^* u\|^2\right)\\
		+O\left(\|D_\tau^{k-j+2}(\nabla^c_\zeta)^{j-2}\bar{\partial} u\|^2+\|D_\tau^{k-j+2}(\nabla^c_\zeta)^{j-2}\bar{\partial}^* u\|^2\right)\\
		+O\left(\|\bar{\partial} u\|_{W^k}\|\bar{\partial} u\|_{W^{k-1}}+\|\bar{\partial}^* u\|_{W^k}\|\bar{\partial}^* u\|_{W^{k-1}}\right).
		\end{multline*}
		Since $\Box\bar{\partial}=\bar{\partial}\Box$ and $\Box\bar{\partial}^*=\vartheta\Box$, we have
		\begin{multline*}
		\|D_\tau^{k-j}(\nabla^c_\zeta)^j\bar{\partial} u\|^2+\|D_\tau^{k-j}(\nabla^c_\zeta)^j\bar{\partial}^* u\|^2\leq O\left(\|\Box u\|^2_{W^{k-1}}\right)\\
		+O\left(\|D_\tau^{k-j+2}(\nabla^c_\zeta)^{j-2}\bar{\partial} u\|^2+\|D_\tau^{k-j+2}(\nabla^c_\zeta)^{j-2}\bar{\partial}^* u\|^2\right)\\
		+O\left(\|\bar{\partial} u\|_{W^k}\|\bar{\partial} u\|_{W^{k-1}}+\|\bar{\partial}^* u\|_{W^k}\|\bar{\partial}^* u\|_{W^{k-1}}\right).
		\end{multline*}
		Proceeding by downward induction on $j$, we see that for any $2\leq j\leq k$, we have
		\begin{multline}
		\label{eq:multiple_normal_derivative_estimate}
		\|D_\tau^{k-j}(\nabla^c_\zeta)^j\bar{\partial} u\|^2+\|D_\tau^{k-j}(\nabla^c_\zeta)^j\bar{\partial}^* u\|^2\leq O\left(\|\Box u\|^2_{W^{k-1}}\right)\\
		+O\left(\|D_\tau^{k-1}\nabla^c_\zeta\bar{\partial} u\|^2+\|D_\tau^{k-1}\nabla^c_\zeta\bar{\partial}^* u\|^2+\|D_\tau^{k}\bar{\partial} u\|^2+\|D_\tau^{k}\bar{\partial}^* u\|^2\right)\\
		+O\left(\|\bar{\partial} u\|_{W^k}\|\bar{\partial} u\|_{W^{k-1}}+\|\bar{\partial}^* u\|_{W^k}\|\bar{\partial}^* u\|_{W^{k-1}}\right).
		\end{multline}

When $j=1$, we use, e.g., Lemma 2.2 in \cite{St10} to write $\nabla^c_\zeta v$ as a linear combination of the coefficients of $\bar\partial v$, $\bar{\partial}^* v$, $v$, and tangential derivatives of $v$, where $v\in C^\infty_{0,\tilde q}(\overline\Omega)\cap\Dom\bar{\partial}^*$ for $\tilde q\in \{q-1,q,q+1\}$.  Hence,
\begin{multline}
\label{eq:one_normal_derivative_estimate}
		\|D_\tau^{k-1}\nabla^c_\zeta\bar{\partial} u\|^2+\|D_\tau^{k-1}\nabla^c_\zeta\bar{\partial}^* u\|^2\leq  O\left(\|D_\tau^{k-1}\bar{\partial}^*\bar{\partial} u\|^2+\|D_\tau^{k-1}\bar\partial\bar{\partial}^* u\|^2\right)\\
		+O\left(\|D_\tau^{k}\bar{\partial} u\|^2+\|D_\tau^{k}\bar{\partial}^* u\|^2\right)
		+O\left(\|\bar{\partial} u\|_{W^k}\|\bar{\partial} u\|_{W^{k-1}}+\|\bar{\partial}^* u\|_{W^k}\|\bar{\partial}^* u\|_{W^{k-1}}\right).
		\end{multline}
		Note that
		\[
		\|D^{k-1}_\tau\Box u\|^2=\|D^{k-1}_\tau\bar{\partial}^*\bar{\partial} u\|^2+2\Re\left<D^{k-1}_\tau\bar{\partial}^*\bar{\partial} u,D^{k-1}_\tau\bar{\partial}\bar{\partial}^* u\right>+\|D^{k-1}_\tau\bar{\partial}\bar{\partial}^* u\|^2.
		\]
		Since
		\begin{multline*}
		-2\Re\left<D^{k-1}_\tau\bar{\partial}^*\bar{\partial} u,D^{k-1}_\tau\bar{\partial}\bar{\partial}^* u\right>\\
		\leq-2\Re\left<D^{k-1}_\tau\bar{\partial} u,\bar{\partial} D^{k-1}_\tau\bar{\partial}\bar{\partial}^* u\right>+O\left(\|\bar{\partial} u\|_{W^{k-1}}\|D^{k-1}_\tau\bar{\partial}\bar{\partial}^* u\|\right)\\
		\leq O\left(\|\bar{\partial} u\|_{W^{k-1}}\|\bar{\partial}^* u\|_{W^k}\right),
		\end{multline*}
		we have
		\[
		\|D^{k-1}_\tau\bar{\partial}^*\bar{\partial} u\|^2+\|D^{k-1}_\tau\bar{\partial}\bar{\partial}^* u\|^2\leq\|D^{k-1}_\tau\Box u\|^2+O\left(\|\bar{\partial} u\|_{W^{k-1}}\|\bar{\partial}^* u\|_{W^k}\right).
		\]
Substituting this in \eqref{eq:one_normal_derivative_estimate}, we have
\begin{multline*}
		\|D_\tau^{k-1}\nabla^c_\zeta\bar{\partial} u\|^2+\|D_\tau^{k-1}\nabla^c_\zeta\bar{\partial}^* u\|^2\leq  O\left(\|D^{k-1}_\tau\Box u\|^2+\|D_\tau^{k}\bar{\partial} u\|^2+\|D_\tau^{k}\bar{\partial}^* u\|^2\right)\\
		+O\left(\|\bar{\partial} u\|_{W^{k-1}}\|\bar{\partial}^* u\|_{W^k}+\|\bar{\partial} u\|_{W^k}\|\bar{\partial} u\|_{W^{k-1}}+\|\bar{\partial}^* u\|_{W^k}\|\bar{\partial}^* u\|_{W^{k-1}}\right).
		\end{multline*}
Combining this with \eqref{eq:multiple_normal_derivative_estimate}, we obtain
		\begin{multline}
		\label{eq:one_or_more_normal_derivative_estimate}
		\|D_\tau^{k-j}(\nabla^c_\zeta)^j\bar{\partial} u\|^2+\|D_\tau^{k-j}(\nabla^c_\zeta)^j\bar{\partial}^* u\|^2\leq O\left(\|\Box u\|^2_{W^{k-1}}+\|D_\tau^{k}\bar{\partial} u\|^2+\|D_\tau^{k}\bar{\partial}^* u\|^2\right)\\
+O\left(\|\bar{\partial} u\|_{W^{k-1}}\|\bar{\partial}^* u\|_{W^k}+\|\bar{\partial} u\|_{W^k}\|\bar{\partial} u\|_{W^{k-1}}+\|\bar{\partial}^* u\|_{W^k}\|\bar{\partial}^* u\|_{W^{k-1}}\right).
		\end{multline}
for all $1\leq j\leq k$.

Turning to derivatives of the form $D_\tau^k$, we observe that since $D_\tau^k$ preserves $\Dom\bar{\partial}^*$, we may integrate by parts and obtain
\[
  \|D_\tau^{k}\bar{\partial} u\|^2+\|D_\tau^{k}\bar{\partial}^* u\|^2\leq\left<D_\tau^k\Box u,D_\tau^k u\right>+O\left(\|\bar{\partial} u\|_{W^{k}}\|u\|_{W^k}+\|\bar{\partial}^* u\|_{W^{k}}\|u\|_{W^k}\right).
\]
Substituting this into \eqref{eq:one_normal_derivative_estimate} gives us
\begin{multline*}
		\|D_\tau^{k-j}(\nabla^c_\zeta)^j\bar{\partial} u\|^2+\|D_\tau^{k-j}(\nabla^c_\zeta)^j\bar{\partial}^* u\|^2\leq O\left(\|\Box u\|^2_{W^{k-1}}+\left<D_\tau^k\Box u,D_\tau^k u\right>\right)\\
+O\left(\|\bar{\partial} u\|_{W^{k-1}}\|\bar{\partial}^* u\|_{W^k}+\|\bar{\partial} u\|_{W^k}\|\bar{\partial} u\|_{W^{k-1}}+\|\bar{\partial}^* u\|_{W^k}\|\bar{\partial}^* u\|_{W^{k-1}}\right)\\
+O\left(\|\bar{\partial} u\|_{W^{k}}\|u\|_{W^k}+\|\bar{\partial}^* u\|_{W^{k}}\|u\|_{W^k}\right).
		\end{multline*}
for all $0\leq j\leq k$.  Summing over all order $k$ differential operators and using a small constant/large constant estimate to absorb terms of the form $\|\bar{\partial}^* u\|_{W^k}$ or $\|\bar{\partial} u\|_{W^k}$ in the left-hand side, we obtain
\begin{multline*}
		\|\bar{\partial} u\|^2_{W^k}+\|\bar{\partial}^* u\|^2_{W^k}\leq O\left(\|\Box u\|^2_{W^{k-1}}+\left<D_\tau^k\Box u,D_\tau^k u\right>\right)\\
+O\left(\|\bar{\partial} u\|_{W^{k-1}}^2+\|\bar{\partial}^* u\|_{W^{k-1}}^2+\|u\|_{W^k}^2\right).
		\end{multline*}
The induction hypothesis will give us \eqref{eq:benign_derivative_estimate}.
		
	\end{proof}
	
	Combining the previous proposition and Lemma 5.6 of \cite{St10}, we obtain the following:
	
	\begin{proposition}
\label{prop:benign_derivatives}
		Let $\Omega$ be a smooth bounded pseudoconvex domain in $\mathbb{C}^m$ with $k\in\mathbb{N}$ and $0\leq q\leq m$. Then there exists a constant $C_k$ such that for $u\in C^\infty_{(0,q)}(\overline{\Omega})\cap\Dom(\Box)$, we have the estimates:
		
		\[\|\nabla_{\overline{L}_i}u\|^2_{k-1}\leq C_k(\|\Box u\|^2_{W^{k-1}}+\|u\|_{W^{k-1}}^2),\]  for $1\leq i\leq m$.
	Moreover,
		\[\|\nabla_{L_i}u\|^2_{k-1}\leq C_k(\|\Box u\|^2_{W^{k-1}}+\|u\|_{W^{k-1}}^2+\|u\|_{W^{k-1}}\|u\|_{W^k}),\]  for $1\leq i\leq m-1$.
	\end{proposition}

\section{Commutation Relations}
	
\subsection{Commutators of First Derivatives}
\label{sec:commutators_first_derivatives}

Since
\[
  \bar{\partial}u=\frac{1}{2^q}{\sum_{I}}' \sum_{i\notin I} \overline{L}_i g(u,\omega_{\overline{L}_I})\omega_{\overline{L}_i}\wedge\omega_{\overline{L}_I}-\frac{1}{2^q}{\sum_{I}}' \sum_{i\notin I} g(u,\nabla_{L_i}\omega_{\overline{L}_I})\omega_{\overline{L}_i}\wedge\omega_{\overline{L}_I},
\]
we have
\[
  [\bar{\partial},\nabla_{L_m-\overline{L}_m}^c]u=\frac{1}{2^q}{\sum_{I}}' \sum_{i\notin I} [\overline{L}_i,L_m-\overline{L}_m] g(u,\omega_{\overline{L}_I})\omega_{\overline{L}_i}\wedge\omega_{\overline{L}_I}+\text{lower order terms}.
\]
To compute the component of $[\overline{L}_i,L_m-\overline{L}_m]$ which is not estimated by Proposition \ref{prop:benign_derivatives}, we first use the fact that $\nabla$ is torsion-free to compute
\[
  g([\overline{L}_i,L_m-\overline{L}_m],L_m)=g(\nabla_{\overline{L}_i}(L_m-\overline{L}_m)-\nabla_{L_m-\overline{L}_m}\overline{L}_i,L_m).
\]
Since $\nabla$ is also the Chern connection, type considerations give us
\[
  g([\overline{L}_i,L_m-\overline{L}_m],L_m)=g(\nabla_{\overline{L}_i}L_m,L_m).
\]
Finally, since $\nabla$ is metric compatible and $g(L_m,L_m)=\frac{1}{2}$,
\[
  g([\overline{L}_i,L_m-\overline{L}_m],L_m)=-g(L_m,\nabla_{L_i}L_m),
\]
and hence
\[[\bar{\partial},\nabla_{L_m-\overline{L}_m}^c]u
=-\frac{1}{2^{q-1}}{\sum_{I}}' \sum_{i\notin I} g(L_m, \nabla_{L_i}L_m)(L_m-\overline{L}_m)g(u,\omega_{\overline{L}_I})\omega_{\overline{L}_i}\wedge\omega_{\overline{L}_I}+K,\]
where $K$ consists of lower order terms and derivatives which can be estimated by Proposition \ref{prop:benign_derivatives}.  In particular, $\|K\|^2\leq \epsilon\|u\|_1^2+C_\epsilon(\|\Box u\|^2+\|u\|^2)$ for any $\epsilon>0$, where $C_\epsilon$ is some positive constant.

Similarly,
\[\bar{\partial}^* u=-\frac{1}{2^{q-1}}{\sum_{I}}'\sum_{i\in I} L_i g(u, \omega_{\overline{L}_I})\overline{L}_i\lrcorner\omega_{\overline{L}_I}+\frac{1}{2^{q-1}}{\sum_{I}}'\sum_{i\in I} g(u, \nabla_{\overline{L}_i}\omega_{\overline{L}_I})\overline{L}_i\lrcorner\omega_{\overline{L}_I},\]
so
\[[\bar{\partial}^*,\nabla_{L_m-\overline{L}_m}^c]u=-\frac{1}{2^{q-1}}{\sum_{I}}'\sum_{i\in I} [L_i,L_m-\overline{L}_m] g(u, \omega_{\overline{L}_I})\overline{L}_i\lrcorner\omega_{\overline{L}_I}+\text{lower order terms}.\]
Observe that $[L_i,L_m]\delta=0$, so $[L_i,L_m]$ is in the span of $\{L_1,\ldots,L_{m-1}\}$.  Hence,
\[
  g([L_i,L_m-\overline{L}_m],L_m)=-g([L_i,\overline{L}_m],L_m).
\]
Using the same techniques as before,
\[
  -g([L_i,\overline{L}_m],L_m)=-g(L_i,\nabla_{L_m}L_m),
\]
and hence
\[	[\bar{\partial}^*,\nabla_{L_m-\overline{L}_m}^c]u
	=\frac{1}{2^{q-2}}{\sum_{I}}' \sum_{i\in I}g(L_i,\nabla_{L_m}L_m)(L_m-\overline{L}_m)g(u, \omega_{\overline{L}_I})\overline{L}_i\lrcorner\omega_{\overline{L}_I}+2R,\]
where $R$ consists of lower order terms and derivatives which can be estimated by Proposition \ref{prop:benign_derivatives}.  In particular, $\|R\|^2\leq \epsilon\|u\|_1^2+C_\epsilon(\|\Box u\|^2+\|u\|^2)$ for any $\epsilon>0$, where $C_\epsilon$ is some positive constant.

Since \[g(\nabla_{L_i}L_m,L_m)=g(\nabla_{L_i}\nabla\delta,L_m)=\Hessian(L_i,L_m)=\overline{\Hessian(L_m,L_i)}=g(L_i,\nabla_{L_m}L_m),\] we obtain that,
\begin{equation}
\label{eq:related coefficients}
  g(L_i,\nabla_{L_m}L_m)=\overline{g(L_m,\nabla_{L_i}L_m)}.
\end{equation}

Furthermore, we consider the vector field $T=e^\psi(L_m-\overline{L}_m)$, where $\psi$ is a real, smooth function defined in a neighborhood of $\partial\Omega$. Since the weight function $e^\psi$ is bounded between two positive numbers, we can see $T$ is comparable with $L_m-\overline{L}_m$. Based on the computation of $[\bar{\partial},\nabla_{L_m-\overline{L}_m}^c]u$, we calculate

	\[[\bar{\partial},e^\psi\nabla^c_{L_m-\overline{L}_m}]u=\frac{1}{2^q}{\sum_{I}}' \sum_{i\notin I} \overline{L}_i(e^\psi)(L_m-\overline{L}_m)g(u,\omega_{\overline{L}_I})\omega_{\overline{L}_i}\wedge\omega_{\overline{L}_I}+e^\psi[\bar{\partial},\nabla^c_{L_m-\overline{L}_m}]u.\]

More precisely, we have that

\[	[\bar{\partial},\nabla^c_{T}]u
	=\frac{1}{2^q}{\sum_{I}}' \sum_{i\notin I} e^\psi(\overline{L}_i\psi-2g(L_m, \nabla_{L_i}L_m))(L_m-\overline{L}_m)g(u,\omega_{\overline{L}_I})\omega_{\overline{L}_i}\wedge\omega_{\overline{L}_I}+Ke^\psi,
\]
where $\|K\|^2\leq \epsilon\|u\|_1^2+C_\epsilon(\|\bar{\partial}u\|^2+\|\bar{\partial}^* u\|^2)$ for arbitrary $\epsilon>0$.

Similarly, we have that

\[	[\bar{\partial}^*,\nabla^c_T]u=-\frac{1}{2^{q-1}}{\sum_{I}}' \sum_{i\in I}e^\psi (L_i\psi -2g(L_i,\nabla_{L_m}L_m)) (L_m-\overline{L}_m) g(u, \omega_{\overline{L}_I})\overline{L}_i\lrcorner\omega_{\overline{L}_I}+2e^\psi R,\]

where $\|2R\|^2\leq \epsilon\|u\|_1^2+C_\epsilon(\|\bar{\partial}u\|^2+\|\bar{\partial}^* u\|^2)$ for arbitrary $\epsilon>0$.

\subsection{Higher Order Derivatives}
In this section, we will consider the commutators for derivatives of higher orders. In other words, we consider the terms of $[\bar{\partial},(\nabla^c_{T})^k]u$ and $[\bar{\partial}^*,(\nabla^c_{T})^k]u$.
\begin{lemma}
	For $k\in\mathbb{N}$,
	\[[\bar{\partial},(\nabla^c_{T})^k]u=\frac{k}{2^q}{\sum_{I}}' \sum_{i\notin I} (\overline{L}_i\psi-2g(L_m, \nabla_{L_i}L_m))T^kg(u,\omega_{\overline{L}_I})\omega_{\overline{L}_i}\wedge\omega_{\overline{L}_I}+K,\]
	where $\|K\|^2\leq \epsilon\|u\|_{W^k}^2+C_\epsilon(\|\Box u\|_{W^{k-1}}^2+\|u\|_{W^{k-1}}^2)$ for arbitrary $\epsilon>0$.
\end{lemma}
\begin{proof}
	We use induction for $k\in\mathbb{N}$.
	We observe that the lemma holds for $k=1$ by the computation in the previous section.
	
	We assume the lemma holds for $(\nabla^c_T)^{k-1}$, and observe that

	\[[\bar{\partial},(\nabla^c_T)^k]
	=[\bar{\partial},(\nabla^c_T)^{k-1}]\nabla_T^c+(\nabla^c_T)^{k-1}[\bar{\partial},\nabla_T^c].\]

	We only need consider the action of $(\nabla^c_T)^{k-1}$ on the derivative term of $[\bar{\partial},\nabla^c_T]$, since otherwise it is a derivative of order $k-1$. Hence, we obtain that
	\[[\bar{\partial},(\nabla^c_{T})^k]u=\frac{k}{2^q}{\sum_{I}}' \sum_{i\notin I} (\overline{L}_i\psi-2g(L_m, \nabla_{L_i}L_m))T^kg(u,\omega_{\overline{L}_I})\omega_{\overline{L}_i}\wedge\omega_{\overline{L}_I}+K,\]
	where $\|K\|^2\leq \epsilon\|u\|_{W^k}^2+C_\epsilon(\|\Box u\|_{W^{k-1}}^2+\|u\|_{W^{k-1}}^2)$ for arbitrary $\epsilon>0$.
\end{proof}
Similarly, we have the lemma for $\bar{\partial}^*$.
\begin{lemma}
	For $k\in\mathbb{N}$,
	\[[\bar{\partial}^*,(\nabla^c_{T})^k]u=-\frac{k}{2^{q-1}}{\sum_{I}}' \sum_{i\in I}(L_i\psi -2g(L_i,\nabla_{L_m}L_m)) T^k g(u, \omega_{\overline{L}_I})\overline{L}_i\lrcorner\omega_{\overline{L}_I}+2R,\]
	where $\|2R\|^2\leq \epsilon\|u\|_{W^k}^2+C_\epsilon(\|\Box u\|_{W^{k-1}}^2+\|u\|_{W^{k-1}}^2)$ for arbitrary $\epsilon>0$.
\end{lemma}

\subsection{Commutators with Cutoff Functions}
\label{sec:commutators_cutoff_functions}
In this section, we will compute the terms $[\bar{\partial},\chi(\nabla^c_{T})^k]u$ and $[\bar{\partial}^*,\chi(\nabla^c_{T})^k]u$, where $\chi$ is a smooth cutoff function. These two terms have practical use to isolate the weakly pseudoconvex points from strongly pseudoconvex points in order to use Theorem \ref{pseudolocal}.

We compute, for $k\in\mathbb{N}$,
	\begin{align*}
	&[\bar{\partial},\chi(\nabla^c_{T})^k]u\\
	=&\frac{1}{2^q}{\sum_{I}}' \sum_{i\notin I} (\overline{L}_i\chi)T^kg(u,\omega_{\overline{L}_I})\omega_{\overline{L}_i}\wedge\omega_{\overline{L}_I}+\chi[\bar{\partial},(\nabla^c_T)^k]u\\
	=&\frac{1}{2^q}{\sum_{I}}' \sum_{i\notin I} (\overline{L}_i\chi)T^kg(u,\omega_{\overline{L}_I})\omega_{\overline{L}_i}\wedge\omega_{\overline{L}_I}\\
&\qquad+\frac{k}{2^q}{\sum_{I}}' \sum_{i\notin I} \chi (\overline{L}_i\psi-2g(L_m, \nabla_{L_i}L_m))T^kg(u,\omega_{\overline{L}_I})\omega_{\overline{L}_i}\wedge\omega_{\overline{L}_I}+\chi K,
	\end{align*}
	where $\|K\|^2\leq \epsilon\|u\|_{W^k}^2+C_\epsilon(\|\Box u\|_{W^{k-1}}^2+\|u\|_{W^{k-1}}^2)$ for arbitrary $\epsilon>0$. Since we assume that $\chi\equiv 1$ on $\Sigma$, we have $\support(\overline{L}_i\chi )\cap \partial\Omega\Subset\partial\Omega\backslash\Sigma$.

Thus, we have that
\begin{align*}
	&\langle[\bar{\partial},\chi(\nabla^c_{T})^k]u,[\bar{\partial},\chi(\nabla^c_{T})^k]u\rangle\\
	\leq &\Big(\frac{k^2(m-q)}{2^{q-1}}{m\choose q}+\epsilon\Big) {\sum_I}'\sum_{i\notin I}\int_{\Omega} |\overline{L}_i\psi-2g(L_m, \nabla_{L_i}L_m)|^2|g(\chi(\nabla_{T}^c)^ku,\omega_{\overline{L}_I})|^2\,dV\\
	&+C_\epsilon\left(\|{\sum_{I}}' \sum_{i\notin I} (\overline{L}_i\chi)T^kg(u,\omega_{\overline{L}_I})\omega_{\overline{L}_i}\wedge\omega_{\overline{L}_I}\|^2+\|\Box u\|_{W^{k-1}}^2+\|u\|_{W^{k-1}}^2\right).
\end{align*}
In this estimate, we have used the following inequality: for a sequence of complex numbers $\lbrace a_j\rbrace_{j=1}^n$, \[\left|\sum_{j=1}^{n}a_j\right|^2\leq n\sum_{j=1}^{n}\left|a_j\right|^2.\] One can see this through the Cauchy--Schwarz inequality.  In this case, we have $n=(m-q){m\choose q}$, since the space of increasing multi-indices of length $q$ has dimension ${m\choose q}$, and for each $I$ the number of elements $i\notin I$ is equal to $m-q$.

Similarly, we have the following.
	\begin{align*}
	&[\bar{\partial}^*,\chi(\nabla^c_{T})^k]u\\
	=&-\frac{1}{2^{q-1}}{\sum_{I}}' \sum_{i\in I}(L_i\chi) T^kg(u, \omega_{\overline{L}_I})\overline{L}_i\lrcorner\omega_{\overline{L}_I}+\chi[\bar{\partial}^*,(\nabla_T^c)^k]u\\
	=&-\frac{1}{2^{q-1}}{\sum_{I}}' \sum_{i\in I}(L_i\chi) T^kg(u, \omega_{\overline{L}_I})\overline{L}_i\lrcorner\omega_{\overline{L}_I}\\
&\qquad-\frac{k}{2^{q-1}}{\sum_{I}}' \sum_{i\in I}\chi (L_i\psi -2g(L_i,\nabla_{L_m}L_m)) T^k g(u, \omega_{\overline{L}_I})\overline{L}_i\lrcorner\omega_{\overline{L}_I}+2\chi R,
	\end{align*}
	where $\|2R\|^2\leq \epsilon\|u\|_{W^k}^2+C_\epsilon(\|\Box u\|_{W^{k-1}}^2+\|u\|_{W^{k-1}}^2)$ for arbitrary $\epsilon>0$.

Thus, we have that
\begin{align*}
&\langle[\bar{\partial}^*,\chi(\nabla^c_{T})^n]u,[\bar{\partial}^*,\chi(\nabla^c_{T})^n]u\rangle\\
\leq&\Big(\frac{k^2 q}{2^{q-1}}{m\choose q}+\epsilon\Big){\sum_{I}}' \sum_{i\in I}\int_{\Omega}|L_i\psi -2g(L_i,\nabla_{L_m}L_m)|^2|g(\chi(\nabla_T^c)^ku, \omega_{\overline{L}_I})|^2\,dV\\
&+C_\epsilon\left(\|{\sum_{I}}' \sum_{i\in I}(L_i\chi) T^kg(u, \omega_{\overline{L}_I})\overline{L}_i\lrcorner\omega_{\overline{L}_I}\|^2+\|\Box u\|_{W^{k-1}}^2+\|u\|_{W^{k-1}}^2\right).
\end{align*}
Therefore, by rescaling $\epsilon$, and taking the support of $\chi$ sufficiently small so that
\[
  \sup_{\support\chi}\sum_{i=1}^m|L_i\psi -2g(L_i,\nabla_{L_m}L_m)|^2\leq m(1+\epsilon)\sup_{i,\Sigma}|L_i\psi -2g(L_i,\nabla_{L_m}L_m)|^2
\]
we may use \eqref{eq:related coefficients} to obtain\begin{equation}\label{main}
\begin{split}
	&\|[\bar{\partial}^*,\chi(\nabla^c_{T})^k]u\|^2+\|[\bar{\partial},\chi(\nabla^c_{T})^k]u\|^2\\
\leq& 2\Big(k^2m^2{m\choose q}+\epsilon\Big)\sup_{i,\Sigma}|L_i\psi -2g(L_i,\nabla_{L_m}L_m)|^2 \|\chi(\nabla^c_{ T})^ku\|^2+\epsilon\|\chi(\nabla_T^c)^k u\|^2\\&+C_\epsilon\|\Box u\|_{W^{k-1}}^2+C_\epsilon\|u\|_{W^{k-1}}^2,
\end{split}
\end{equation}

because $L_i\chi$ and $\overline{L}_i\chi$ are only nonzero at strongly pseudoconvex points and Theorem \ref{pseudolocal} applies there.

\section{The \textit{A Priori} Estimate}

In this section we will show how to use the commutators in previous sections to prove our main theorem. We will make repeated use of the following estimate: for any $v,f\in C^\infty_{0,q}(\overline\Omega)$ supported in a coordinate boundary chart,
\begin{equation}
\label{eq:integration_by_parts_k}
  \left|\left<\chi(\nabla^c_T)^k v,f\right>-\left<v,\chi(\nabla^c_T)^k f\right>\right|\leq O\left(\left<(\nabla^c_T)^{k-1} v,f\right>\right).
\end{equation}
We will prove this by induction on $k$.  Since $\nabla^c_T$ is compatible with the metric $g$ and $T$ is purely imaginary,
\[
  \left<\chi(\nabla^c_T)^k v,f\right>=\int_\Omega\chi T g((\nabla^c_T)^{k-1}v,f)dV+\left<\chi(\nabla^c_T)^{k-1} v,\nabla^c_T f\right>.
\]
Since $T$ is tangential, Stokes' Theorem will imply
\[
  \left|\left<\chi(\nabla^c_T)^k v,f\right>-\left<\chi(\nabla^c_T)^{k-1} v,\nabla^c_T f\right>\right|\leq O\left(\left<(\nabla^c_T)^{k-1}v,f\right>\right).
\]
This proves \eqref{eq:integration_by_parts_k} when $k=1$.  When $k\geq 2$, we use our induction hypothesis to obtain
\[
  \left|\left<\chi(\nabla^c_T)^{k-1} v,\nabla^c_T f\right>-\left<v,\chi(\nabla^c_T)^{k} f\right>\right|\leq O\left(\left<(\nabla^c_T)^{k-2} v,\nabla^c_T f\right>\right).
\]
Applying \eqref{eq:integration_by_parts_k}  with $k=1$ to the error term, we have \eqref{eq:integration_by_parts_k}.

Consider $\|\chi(\nabla^c_T)^ku\|^2$.
The following estimate is essentially contained in \cite{CS01} and the reader is referred to \cite{St10} as well.  From Theorem \ref{basic}, we obtain for any $u\in C^\infty_{0,q}(\overline\Omega)\cap\Dom\bar{\partial}^*$ supported in a special boundary chart:
\begin{align*}
	&\langle \chi(\nabla^c_T)^ku, \chi(\nabla^c_T)^ku\rangle\\
	\leq&B\langle\bar{\partial}\chi(\nabla^c_T)^ku,\bar{\partial}\chi(\nabla^c_T)^ku\rangle+B\langle\bar{\partial}^* \chi(\nabla^c_T)^ku,\bar{\partial}^* \chi(\nabla^c_T)^ku\rangle\\
	=&B\langle\chi(\nabla^c_T)^k\bar{\partial}u,\bar{\partial}\chi(\nabla^c_T)^ku\rangle+B\langle[\bar{\partial}, \chi(\nabla^c_T)^k]u,\bar{\partial}\chi(\nabla^c_T)^ku\rangle\\&+B\langle\chi(\nabla^c_T)^k\bar{\partial}^* u,\bar{\partial}^* \chi(\nabla^c_T)^ku\rangle+B\langle[\bar{\partial}^*, \chi(\nabla^c_T)^k] u,\bar{\partial}^* \chi(\nabla^c_T)^ku\rangle
\end{align*}
Using \eqref{eq:integration_by_parts_k}, we have:
\begin{align*}
	&B\langle\bar{\partial}\chi(\nabla^c_T)^ku,\bar{\partial}\chi(\nabla^c_T)^ku\rangle+B\langle\bar{\partial}^* \chi(\nabla^c_T)^ku,\bar{\partial}^* \chi(\nabla^c_T)^ku\rangle\\
	\leq &B\langle\bar{\partial}u, \chi (\nabla^c_T)^k\bar{\partial}\chi(\nabla^c_T)^ku\rangle+B\langle[\bar{\partial}, \chi(\nabla^c_T)^k]u,\bar{\partial}\chi(\nabla^c_T)^ku\rangle+B\langle\bar{\partial}^* u,\chi(\nabla^c_T)^k\bar{\partial}^*\chi(\nabla^c_T)^ku\rangle\\&+B\langle[\bar{\partial}^*,\chi(\nabla^c_T)^k] u,\bar{\partial}^* \chi(\nabla^c_T)^ku\rangle+O\left(\langle(\nabla^c_T)^{k-1}\bar{\partial}u,\bar{\partial}\chi(\nabla^c_T)^ku\rangle+\langle(\nabla^c_T)^{k-1}\bar{\partial}^* u,\bar{\partial}^* \chi(\nabla^c_T)^ku\rangle\right)
\end{align*}
If $u\in\Dom\Box$ as well, then we may take additional commutators and integrate by parts to obtain:
\begin{align*}
	&B\langle\bar{\partial}\chi(\nabla^c_T)^ku,\bar{\partial}\chi(\nabla^c_T)^ku\rangle+B\langle\bar{\partial}^* \chi(\nabla^c_T)^ku,\bar{\partial}^* \chi(\nabla^c_T)^ku\rangle\\
	\leq&B\langle\bar{\partial}u, \chi(\nabla^c_T)^k[\chi(\nabla^c_T)^k,\bar{\partial}]u\rangle+B\langle\bar{\partial}u, [[\chi(\nabla^c_T)^k,\bar{\partial}],\chi(\nabla^c_T)^k]u\rangle+B\langle\bar{\partial}^* u,\chi(\nabla^c_T)^k[\chi(\nabla^c_T)^k,\bar{\partial}^*]u\rangle\\&+B\langle\bar{\partial}^* u,[[\chi(\nabla^c_T)^k,\bar{\partial}^*],\chi(\nabla^c_T)^k]u\rangle+B\langle\Box u, \chi(\nabla^c_T)^k\chi(\nabla^c_T)^ku\rangle+B\langle[\bar{\partial}, \chi(\nabla^c_T)^k]u,\bar{\partial}\chi(\nabla^c_T)^ku\rangle\\&+B\langle[\bar{\partial}^*,\chi(\nabla^c_T)^k] u,\bar{\partial}^* \chi(\nabla^c_T)^ku\rangle+O\left(\langle(\nabla^c_T)^{k-1}\bar{\partial}u,\bar{\partial}\chi(\nabla^c_T)^ku\rangle+\langle(\nabla^c_T)^{k-1}\bar{\partial}^* u,\bar{\partial}^* \chi(\nabla^c_T)^ku\rangle\right).
\end{align*}
We now use \eqref{eq:integration_by_parts_k} again and take further commutators to obtain:
\begin{align*}
&B\langle\bar{\partial}\chi(\nabla^c_T)^ku,\bar{\partial}\chi(\nabla^c_T)^ku\rangle+B\langle\bar{\partial}^* \chi(\nabla^c_T)^ku,\bar{\partial}^* \chi(\nabla^c_T)^ku\rangle\\
	\leq&B\langle\chi(\nabla^c_T)^k\Box u, \chi(\nabla^c_T)^ku\rangle+2B\i\Im\langle[\bar{\partial}, \chi(\nabla^c_T)^k]u,\bar{\partial}\chi(\nabla^c_T)^ku\rangle\\
&+2B\i\Im\langle[\bar{\partial}^*,\chi( \nabla^c_T)^k] u,\bar{\partial}^* \chi(\nabla^c_T)^ku\rangle+B\langle[\bar{\partial},\chi(\nabla^c_T)^k]u, [\bar{\partial},\chi(\nabla^c_T)^k]u\rangle\\
&+B\langle\bar{\partial}u, [[\chi(\nabla^c_T)^k,\bar{\partial}],\chi(\nabla^c_T)^k]u\rangle+B\langle[\bar{\partial}^*,\chi(\nabla^c_T)^k ]u,[\bar{\partial}^*,\chi(\nabla^c_T)^k]u\rangle\\
&+B\langle\bar{\partial}^* u,[[\chi(\nabla^c_T)^k,\bar{\partial}^*],\chi(\nabla^c_T)^k]u\rangle+O\left(\langle\chi(\nabla^c_T)^{k-1}\Box u, \chi(\nabla^c_T)^ku\rangle\right)\\
&+O\left(\langle\chi(\nabla^c_T)^{k-1}\bar{\partial}u, [\chi(\nabla^c_T)^k,\bar{\partial}]u\rangle+\langle\chi(\nabla^c_T)^{k-1}\bar{\partial}^* u,[\chi(\nabla^c_T)^k,\bar{\partial}^*]u\rangle\right)\\ &+O\left(\langle(\nabla^c_T)^{k-1}\bar{\partial}u,\bar{\partial}\chi(\nabla^c_T)^ku\rangle+\langle(\nabla^c_T)^{k-1}\bar{\partial}^* u,\bar{\partial}^* \chi(\nabla^c_T)^ku\rangle\right).
\end{align*}
Observe that this expression is necessarily real, so we may discard the imaginary part of each term, especially those terms which are purely imaginary.  For all the remaining terms of the form $\left<\cdot,\bar{\partial}\chi(\nabla^c_T)^ku\right>$ or $\left<\cdot,\bar{\partial}^*\chi(\nabla^c_T)^ku\right>$, we use a small constant/large constant inequality to absorb one term on the left-hand side and divide by the new coefficient.  Replacing the left-hand side with $\|\chi(\nabla^c_T)^ku\|^2$ (using Theorem \ref{basic}) and using a similar small constant/large constant inequality allows us to take care of terms of the form $\left<\cdot,\chi(\nabla^c_T)^ku\right>$.  As noted in \cite{St10} (see the discussion after (3.57)), the nested commutators can be expressed as the composition of a tangential $(k-1)$st order operator with a $k$th order operator, so we may integrate by parts with the $k-1$ tangential derivatives.  Using small constant/large constant inequalities for the remaining terms, we are left with
\begin{align*}
&\|\chi(\nabla^c_T)^ku\|^2\\
	\leq&(B+\epsilon)\|[\bar{\partial},\chi(\nabla^c_T)^k]u\|^2
+(B+\epsilon)\|[\bar{\partial}^*,\chi(\nabla^c_T)^k ]u\|^2+\epsilon\|u\|^2_{W^k}\\
&+C_\epsilon\left(\|\Box u\|^2_{W^k}+\|\bar{\partial}u\|^2_{W^{k-1}}+\|\bar{\partial}^* u\|^2_{W^{k-1}}\right),
\end{align*}
for any $\epsilon>0$, where $C_\epsilon>0$.

We proceed by induction, assuming $\|u\|_{W^{k-1}}\lesssim\|\Box u\|_{W^{k-1}}$.  Using Proposition \ref{prop:benign_derivative_setup}, we have:
\[\|\chi(\nabla^c_T)^ku\|^2\leq C_{\epsilon,k}\|\Box u\|^2_{W^{k}}+\epsilon\|u\|^2_k+(B+\epsilon)\|[\bar{\partial},\chi(\nabla^c_T)^k]u\|^2+(B+\epsilon)\|[\bar{\partial}^*,\chi(\nabla^c_T)^k]u\|^2\]

Similarly, we have that, by the estimate (\ref{main}) and rescaling $C_{k,\epsilon}$ and $\epsilon$,
\begin{equation}\label{final}
	\begin{split}
	&\|\chi(\nabla^c_T)^ku\|^2\\
	\leq &C_{k,\epsilon}\|\Box u\|^2_{W^{k}}+\epsilon\|u\|^2_k+2\left(Bk^2m^2{m\choose q}+\epsilon\right)\sup_{i,\Sigma}|L_i\psi -2g(L_i,\nabla_{L_m}L_m)|^2 \|\chi(\nabla^c_{ T})^ku\|^2.
	\end{split}
\end{equation}

	We obtain the following theorem.
	
	\begin{theorem}\label{maintheorem}
		Let $\Omega$ be a bounded pseudoconvex domain with smooth boundary in $\mathbb{C}^m$ and $1\leq q\leq m$. Assume there exists a smooth, real-valued function defined in a neighborhood $U$ of $\Sigma$ in $\partial\Omega$ so that \[\sup_{\Sigma}|L\psi-2g(\nabla_{L}L_m,L_m)|<\frac{1}{km\sqrt{B{m\choose q}}},\] for all unit $(1,0)$ vector fields $L$ in a neighborhood of $\partial\Omega$ where $k\in\mathbb{N}$ and $B$ is the constant from Theorem \ref{basic}. Then for any smooth $(0,q)$-form u in $\Dom(\Box)$, we have \[\|u\|_{W^k}\lesssim\|\Box u\|_{W^k}\]
	\end{theorem}
\begin{proof}
	We assume $u\in C^\infty_{(0,q)}(\overline{\Omega})\cap\Dom(\Box)$.
	In a local coordinate chart, this follows from the estimate (\ref{final}) by induction on $k$ and Proposition \ref{prop:benign_derivatives}. For the general case, we choose a partition of unity $\lbrace\lambda_\alpha\rbrace_{\alpha=1}^M$ such that $\support \lambda_\alpha\subset U_\alpha$. Although $\nabla_{T,\alpha}^c$ depends on our choice of local coordinates, we recall that $\nabla_{T,\alpha}^c-\nabla_{T}$ is a zero-order operator on $U_\alpha$, so $(\nabla_{T,\alpha}^c)^k-(\nabla_T)^k$ is an operator of order $k-1$ on $U_\alpha$.  We observe that for $\mathcal{M}\subset\{1,\ldots,M\}$ such that $\bigcap_{\alpha\in\mathcal{M}}U_\alpha\neq\emptyset$, $\sum_{\alpha\in\mathcal{M}}\bar{\partial}\lambda_\alpha=0$, so
\[
  \sum_{\alpha\in\mathcal{M}}[\bar{\partial},\chi\lambda_\alpha(\nabla^c_{T,\alpha})^k]u
  =\sum_{\alpha\in\mathcal{M}}\lambda_\alpha[\bar{\partial},\chi(\nabla^c_{T,\alpha})^k]u+\sum_{\alpha\in\mathcal{M}}(\bar{\partial}\lambda_\alpha)\wedge\chi((\nabla^c_{T,\alpha})^k-(\nabla_T)^k)u.
\]
As computed in Section \ref{sec:commutators_cutoff_functions}, the principal part of $[\bar{\partial},\chi\lambda_\alpha(\nabla^c_{T,\alpha})^k]u$ in the direction of $T$ is independent of $\alpha$, so the principal part of $\sum_{\alpha\in\mathcal{M}}[\bar{\partial},\chi\lambda_\alpha(\nabla^c_{T,\alpha})^k]u$ in the direction of $T$ is also independent of the choice of partition of unity.  The commutator with $\bar{\partial}^*$ is handled similarly.
	
	The terms $g([\bar{\partial},\chi(\nabla^c_T)^k] u,[\bar{\partial},\chi(\nabla^c_T)^k] u)$ and $g([\bar{\partial}^*,\chi(\nabla^c_T)^k] u,[\bar{\partial},\chi(\nabla^c_T)^k] u)$ remain the same except for benign derivatives (which can be estimated by Proposition \ref{prop:benign_derivatives}) or lower order derivatives (which can be estimated by induction), so the estimates in local coordinate charts still hold globally. Thus, we have that
	\[\|\chi(\nabla^c_T)^ku\|\lesssim\|\Box u\|_{W^k}.\] This completes the proof
	
\end{proof}

\section{Beyond the \textit{A Priori} Estimate}
In this section, we will see how to pass Theorem \ref{maintheorem} to the genuine estimate. Our idea is similar to the classical one which the reader can find in \cite{CS01} and \cite{St10}.

Let $\delta$ denote the signed distance function. Let $M>0$ be so large that $\rho_\tau:=\delta+\tau e^{M|z|^2}$ defines a smooth strongly pseudoconvex domain $\Omega_\tau\subset\subset\Omega$ for small $\tau>0$. Let $N_\tau$ denote the $\bar{\partial}$-Neumann operator on $\Omega_\tau$, i.e., $N_\tau=\Box_\tau^{-1}$.  By the classical work of Kohn on strongly pseudoconvex domains with smooth boundaries \cite{Ko63, Ko64}, we know that $N_\tau$ is continuous in $W^k(\Omega_\tau)$ for all $k\in\mathbb{N}$.  Let $u\in W^k(\Omega)$. On each $\Omega_\tau$, we find $N_\tau u\in W^k(\Omega_\tau)$ because of the global regularity of $N_\tau$ on $\Omega_\tau$. By the discussion in previous sections for weakly pseudoconvex points and Theorem \ref{pseudolocal}, we obtain that \[\|N_\tau u\|_{W^k(\Omega_\tau)}\leq C\|u\|_{W^k(\Omega_\tau)}\leq C\|u\|_{W^k(\Omega)},\] where $C$ is independent from $\tau$. Extend $N_\tau u$ to $\overline{\Omega}$ by letting it be zero outside $\Omega_\tau$. We have that \[\|N_\tau u\|_{L^2(\Omega)}\leq\|N_\tau u\|_{W^k(\Omega_\tau)}\leq C\|u\|_{W^k(\Omega)}.\]

Since $\|N_\tau u\|_{L^2(\Omega)}$ is uniformly bounded, then a subsequence of $\lbrace N_\tau u\rbrace_\tau$ is weakly convergent in the $L^2$ norm. Let $v$ be the weak limit.

We want to show that $v\in\Dom(\Box)$. First, to show $v\in\Dom(\bar{\partial}^*)$, we check that for an arbitrary $(0,1)$-form $\phi $ on $\Omega$ in $\Dom(\bar{\partial})$,
\[\langle\bar{\partial}^* v, \phi\rangle:=\langle v, \bar{\partial}\phi\rangle_{L^2(\Omega)}=\lim_\tau\langle N_\tau u, \bar{\partial}\phi\rangle_{L^2(\Omega)}=\lim_\tau\langle \bar{\partial}^*_\tau N_\tau u, \phi\rangle_{L^2(\Omega_\tau)}\]
because $N_\tau u$ is in $\Dom(\Box_\tau)$ for $\Omega_\tau$, where $\bar{\partial}^*_\tau$ is the adjoint of $\bar{\partial}$ in $\Omega_\tau$. But \begin{align*}&\lim_\tau\langle \bar{\partial}^*_\tau N_\tau u, \phi\rangle_{L^2(\Omega_\tau)}\\
&\leq \lim_\tau\|\bar{\partial}^*_\tau N_\tau u\|_{L^2(\Omega_\tau)}\|\phi\|_{L^2(\Omega)}\leq \lim_\tau\sqrt{\|u\|_{L^2(\Omega)}\|N_\tau u\|_{L^2(\Omega)}}\|\phi\|_{L^2(\Omega)}\lesssim\|\phi\|_{L^2(\Omega)}.\end{align*} Thus, by the Hahn--Banach theorem and the Riesz representation theorem, we have that $v\in\Dom(\bar{\partial}^*)$.

We now show $v\in\Dom(\bar{\partial})$ and will also show that $\bar{\partial}v\in\Dom(\bar{\partial}^*)$. For an arbitrary smooth $(0,1)$-form $\phi$ with compact support in $\Omega$, we compute,
\[\langle\bar{\partial}v, \phi\rangle:=\langle v, \bar{\partial}^*\phi\rangle_{L^2(\Omega)}=\lim_\tau\langle N_\tau u, \bar{\partial}^*_\tau\phi\rangle_{L^2(\Omega_\tau)}=\lim_\tau\langle \bar{\partial}N_\tau u, \phi\rangle_{L^2(\Omega_\tau)}=\langle g,\phi\rangle_{L^2(\Omega)},\] where $g$ is the weak limit of $\bar{\partial}N_\tau u$. The weak limit exists because of the uniform boundedness of $\|\bar{\partial}N_\tau u\|_{W^k}$ (see Lemma 3.2 in \cite{St10}).  Thus, $\bar{\partial}v=g\in L^2(\Omega)$, i.e.,  $v\in\Dom(\bar{\partial})$. Moreover, as a byproduct, we can see that \[\langle\bar{\partial}v, \psi\rangle=\lim_\tau\langle\bar{\partial}N_\tau u, \psi\rangle_{L^2(\Omega_\tau)},\] for an arbitrary $(0,1)$-form $\psi$ in $L^2(\Omega)$. Thus, for an arbitrary $(0,1)$-form $\phi$ on $\Omega$ in $\Dom(\bar{\partial})$,  \[\langle\bar{\partial}v, \bar{\partial}\phi\rangle_{L^2(\Omega)}=\lim_\tau\langle\bar{\partial}N_\tau u, \bar{\partial}\phi\rangle_{L^2(\Omega_\tau)}=\lim_\tau\langle \bar{\partial}^*_\tau\bar{\partial}N_\tau u, \phi\rangle_{L^2(\Omega_\tau)}\leq\|u\|_{L^2(\Omega)}\|\phi\|_{L^2(\Omega)}.\] Again, by the Hahn--Banach theorem and the Riesz representation theorem, we have that $\bar{\partial}v\in\Dom(\bar{\partial}^*)$. And hence, $v\in\Dom(\Box)$.

We now show $\Box v=u$. For an arbitrary $(0,1)$-form $\phi$ with compact support in $\Omega$, we verify \[\langle\Box v, \phi\rangle_{L^2(\Omega)}=\langle u,\phi\rangle_{L^2(\Omega)}.\]

Observe that \[\langle\Box v, \phi\rangle_{L^2(\Omega)}=\langle v, \Box \phi\rangle_{L^2(\Omega)}=\lim_\tau\langle N_\tau u, \Box\phi\rangle_{L^2(\Omega)}=\langle u, \phi\rangle_{L^2(\Omega)},\] which proves $\Box v=u$ in the distribution sense.  Since we have already shown $v\in\Dom\Box$, we have $\Box v=u$.

We now turn to show that $\|v\|_{W^k(\Omega)}\leq C\|u\|_{W^k(\Omega)}$. Fix an arbitrary $\tau_0>0$ and on $\Omega_{\tau_0}$, $N_\tau u$ weakly converges to $v|_{\Omega_{\tau_0}}$ in $W^k(\Omega_{\tau_0})$. Since \[\|v\|_{W^k(\Omega_{\tau_0})}\leq \liminf_\tau\|N_\tau u\|_{W^k(\Omega_{\tau_0})},\] we obtain that \[\|v\|_{W^k(\Omega_{\tau_0})}\leq C\|u\|_{W^k(\Omega)},\] which implies  \[\|v\|_{W^k(\Omega)}\leq C\|u\|_{W^k(\Omega)},\] because $\tau_0$ is arbitrary.

Thus, for an $u\in W^k(\Omega)$, we can find $v\in W^k(\Omega)\cap\Dom(\Box)$ so that $\Box v=u$ and $\|v\|_{W^k(\Omega)}\leq C\|u\|_{W^k(\Omega)}$. So we obtain Theorem \ref{maintheorem2}.

\bigskip
\bigskip

\printbibliography

\end{document}